\documentclass[12pt,a4paper]{amsart}         
\usepackage[margin=2cm]{geometry}
\usepackage[backend=bibtex,url=false,doi=true,isbn=false,giveninits=true,style=numeric,sorting=nyt]{biblatex}
\usepackage{amsmath,amsfonts,amssymb,amsthm}
\usepackage{mathrsfs,epsfig,epstopdf,url}
\usepackage{enumerate}
\usepackage[table]{xcolor}
\usepackage{multirow}
\usepackage{algorithm}
\usepackage{algpseudocode}

\theoremstyle{definition}
\newtheorem{definition}{Definition}%

\theoremstyle{plain}
\newtheorem{theorem}[definition]{Theorem}
\newtheorem{lemma}[definition]{Lemma}
\newtheorem{proposition}[definition]{Proposition}
\newtheorem*{problem}{Problem}

\newcommand{\GF}[1]{\mathsf{GF}(#1)}

\newcommand{\E}{\mathsf{E}}
\newcommand{\D}{\mathsf{D}}

\newcommand{\Aut}{\mathrm{Aut}}

\newcommand{\K}{\mathcal{K}}
\newcommand{\PS}{\mathcal{P}\mathcal{S}}
\newcommand{\nl}{\mathrm{nl}}
\newcommand{\dA}[1]{d_H(#1,\mathcal{A})}

\title{On the minimum Hamming distance between vectorial Boolean and affine functions}

\author{G\'abor P. Nagy}
\address{Institute of Mathematics \\
	Budapest University of Technology and Economics\\
	M\H{u}egyetem rkp 3\\
	H-1111 Budapest, Hungary}
\address{Bolyai Institute \\
	University of Szeged \\
	Aradi v\'ertan\'uk tere 1\\
	H-6720 Szeged, Hungary}
\email{nagyg@math.u-szeged.hu}

\thanks{Support provided from the Program of Excellence TKP2021-NVA-02 at the Budapest University of Technology and Economics.}


\begin{document}

\begin{abstract}
In this paper, we study the Hamming distance between vectorial Boolean functions and affine functions. This parameter is known to be related to the non-linearity and differential uniformity of vectorial functions, while the calculation of it is in general difficult. In 2017, Liu, Mesnager and Chen conjectured an upper bound for this metric. We prove this bound for two classes of vectorial bent functions, obtained from finite quasigroups in characteristic two, and we improve the known bounds for two classes of monomial functions of differential uniformity two or four. For many of the known APN functions of dimension at most nine, we compute the exact distance to affine functions.
\end{abstract}

\maketitle

\section{Introduction}

Vectorial Boolean functions $f:\GF{2}^n\to \GF{2}^m$, also called \textit{substitution boxes,} play a central role in symmetric key block ciphers. Because they are the only non-linear components of the ciphers, they provide \textit{confusion.} The study of the nonlinear properties of vectorial Boolean functions is fundamental for the evaluation of the resistance of the block cipher against the main attacks, such as the differential attack and the linear attack; see \cite{Matsui1994, Biham1991}. The two main metrics of these nonlinear properties are the \textit{differential uniformity $\delta_f$} (the lower is the less linear), and the \textit{nonlinearity $\nl(f)$.} The \textit{Walsh-Hadamard transform} provides an effective tool for computation with nonlinearity $\nl(f)$. 

In this paper, we deal with a third metric, which is called by different names in the literature: \textit{vectorial nonlinearity} $\mathrm{NL}_{\boldmath{v}}(f)$ by the author of this paper \cite{Nagy2022}, \textit{distance to affine functions} $d_H(f,\mathcal{A})$ by Carlet \cite{Carlet2021b}, \textit{second type of nonlinearity} $\mathcal{N}_\text{v}$ by Liu, Mesnager, and Chen \cite{Liu2017}. This metric should not be confused with the notion of \textit{$r$th-order nonlinearity} of Boolean functions, which is the minimum Hamming distance of $f$ from the set of all $n$-variable Boolean functions having an algebraic degree at most $r$, see \cite{Saini2024}. 

In this paper, we use the terminology ``distance to affine function''. For a map $f:\GF{2}^n\to \GF{2}^m$, we define $\dA{f}$ as the Hamming distance of $f$ to the set $\mathcal{A}$ of affine maps $\GF{2}^n\to \GF{2}^m$. The computation of $\dA{f}$ is generally difficult. In addition, it is difficult to give non-trivial lower and upper bounds for $\dA{f}$. The distance to affine functions is linked to the nonlinearity by the inequality
\begin{align*} %
\dA{f} \geq \nl(f),
\end{align*}
that holds for all $f:\GF{2}^n\to \GF{2}^m$; see \cite{Liu2017}. For recent results on bounds involving nonlinearity and differential uniformity, see \cite{Charpin2019,Xu2018}. Using the concept of Sidon sets and $t$-thin sets, the lower bound
\begin{align} \label{eq:NGP-bound}
\dA{f}\geq 2^n - \sqrt{\delta_f}\cdot 2^{n/2}-\frac{1}{2}
\end{align}
has been shown in \cite{Nagy2022} by the author of this paper, and independently by Ryabov \cite{Ryabov2023}. Concerning upper bounds, we have the trivial inequality
\begin{align*} %
\dA{f}\leq 2^n-n-1.
\end{align*}
The function $f:\GF{2}^n\to \GF{2}^m$ is called an $(n,m)$-bent function, if its nonlinearity achieves the the covering radius upper bound:
\begin{align*} %
\nl(f)=2^{n-1}-2^{n/2-1}.
\end{align*}
For the class of $(n,m)$-bent functions, Carlet, Ding and Yuan \cite{Carlet2005} proved
\begin{align} \label{eq:CDY-2005}
\left(1-\frac{1}{2^m}\right)\big(2^n-2^{n/2}\big)\leq \dA{f} \leq \left(1-\frac{1}{2^m}\right)\big(2^n+2^{n/2}\big).
\end{align}
In 2017, Liu, Mesnager and Chen formulated a conjecture:
\begin{problem}[Liu-Mesnager-Chen Conjecture \cite{Liu2017}]
For any map $f:\GF{2}^n\to \GF{2}^m$, 
\begin{align} \label{eq:LMCconj}
\dA{f}\leq \left(1-\frac{1}{2^m}\right)\big(2^n-2^{n/2}\big)
\end{align}
holds. In particular, if $n=m$, then
\begin{align} \label{eq:LMCconj2APN}
\dA{f}\leq 2^n-2^{n/2}-1.
\end{align}
\end{problem}
The covering radius bound implies the conjecture for $m=1$. For $(n,m)$-bent functions, the conjecture implies the exact value of the distance from affine functions:
\begin{align} \label{eq:bent-true}
\dA{f}=\left(1-\frac{1}{2^m}\right)\big(2^n-2^{n/2}\big).
\end{align}

\begin{theorem} \label{thm:main-bent}
Let $t\leq m$ be positive integers and $(\GF{2^m},+,\star)$ be a left pre-quasifield. Let the maps $\gamma, \tau, h:\GF{2^m}\to \GF{2^t}$ and $\sigma:\GF{2^m}\to \GF{2^m}$ be such that $\gamma$ is balanced, $\tau$ is surjective linear, $\sigma$ is invertible, and $h$ is arbitrary. Define the $(2m,t)$-bent functions $f_1,f_2$ as
\begin{align*} 
f_1(x,y)&=\gamma\left(\star\frac{x}{y}\right), \\
f_2(x,y)&=\tau\left(\sigma(y) \star{x}\right) +h(y).
\end{align*}
Then
\[\dA{f_1} = \dA{f_2} = \left(1-\frac{1}{2^t}\right)\big(2^{2m}-2^{m}\big).\]
\end{theorem}

\begin{theorem} \label{thm:main-dH-bounds}
Let $n$ be a positive integer.
\begin{enumerate}[(i)]
\item If $n$ is even and $h_1(x)=x^{2^n-2}$, then 
\begin{align*}
2^n-\sqrt{2}\cdot 2^{n/2}-\frac{3}{2} \leq \dA{h_1}\leq 2^n-2^{n/2}-2.
\end{align*}
\item If $n$ is divisible by $4$, $d=2^{n/2-1}+1$, and $h_2(x)=x^d$, then
\begin{align*}
2^n-\sqrt{2}\cdot 2^{n/2}-\frac{1}{2} \leq \dA{h_2}\leq 2^n-2^{n/2}-2.
\end{align*}
\end{enumerate}
In particular, the Liu-Mesnager-Chen Conjecture holds for both functions $h_1(x)$, $h_2(x)$. 
\end{theorem}

Notice that $h_1(x)$ is Nyberg's Field Inverse Function, used in the AES cryptosystem with $n=8$. In even dimension, the Field Inverse Function has differential uniformity $4$. The function $h_2(x)$ is a monomial APN function of Gold type, and the exponent is $d=2^k+1$, where $k=n/2-1$. As $n/2$ is even, $k$ is odd and $\gcd(k,n)=\gcd(n/2-1,n)=1$. The APN property of $h_2(x)$ follows. 

At the end of the paper, we will present new results for the minimum Hamming distance between affine and APN functions of dimension $n=6,7,8,9$.

\section{Preliminaries}

Let $A,B$ be finite sets. The function $f:A\to B$ is \textit{balanced,} if the preimage $f^{-1}(b)$ has size $|A|/|B|$ for all $b\in B$. The \textit{Hamming distance} $d_H(f,g)$ between two functions $f,g: A\to B$ is 
defined as the number of inputs at which they differ:
\[ d_H(f,g) = |\{a\in A : f(a) \neq g(a)\}|. \]
We say that a function $f$ is \textit{linear,} if its definition domain is a vector space over $\GF{2}$, and the function is $\GF{2}$-linear. The sum of a linear function and a constant is called an \textit{affine function.} The space of affine functions $\GF{2}^n\to \GF{2}^m$ is denoted by $\mathcal{A}_{n,m}$, or simply $\mathcal{A}$. 

Let $q$ be a prime power, and $t$ a positive integer. For $x,y\in \GF{q}^t$, $x\cdot y$ denotes the (non-degenerate) \textit{inner product} $x_1y_1+\cdots+x_ty_t$. If we identify the $\GF{q}$-vector spaces $\GF{q^t}$ and $\GF{q}^t$, then the inner product can be written as $\mathrm{Tr}_{\GF{q^t}/\GF{q}}(xy)$. The map
\[\mathrm{Tr}_{\GF{q^t}/\GF{q}}(x)=x+x^q+\cdots+x^{q^{t-1}}\]
is the \textit{trace map} of the field extension $\GF{q^t}/\GF{q}$. 

The \textit{nonlinearity} of a Boolean function $f:\GF{2}^n \to \GF{2}$ is defined as the minimum Hamming distance between $f$ and all affine Boolean functions in $n$ variables, denoted by $d_H(f,\mathcal{A})$, where $\mathcal{A}$ is the set of all affine functions on $n$ variables. The nonlinearity of a Boolean function can be expressed as:
\[\nl(f) = d_H(f,\mathcal{A}) = 2^{n-1} - \frac{1}{2} \max_{a \in \GF{2}^n} |W_f(a)|,\]
where $W_f(a)$ is the \textit{Walsh-Hadamard transform} of $f$:
\[W_f(a) = \sum_{x \in \GF{2}^n} (-1)^{f(x)+a \cdot x}, \qquad a \in \GF{2}^m.\]
Parseval's relation $\sum_{a \in \GF{2}^n} W_f^2(a)=2^{2n}$ implies 
\[\nl(f)\leq 2^{n-1}-2^{n/2-1}\]
for every $n$-variable Boolean function $f$. This bound, tight for every even $n$, is called the \textit{covering radius bound.} Functions achieving this bound are called \textit{$n$-variable bent functions.} The \textit{Maiorana-McFarland class} is the most important direct construction of bent functions. Let $\pi$ be an arbitrary permutation on the set $\GF{2}^{n/2}$, and let $h$ be an arbitrary Boolean function in $n/2$ variables, $x,y\in \GF{2}^{n/2}$. Then 
\[f(x,y)=x\cdot \pi(y)+ h(y)\]
is a bent function in $n$ variables. 

\textit{Vectorial Boolean functions,} also called \textit{$(n,m)$-functions,} are functions from $\GF{2}^n \to \GF{2}^m$. If $v\in \GF{2}^m\setminus\{0\}$, then the Boolean function $v\cdot f:\GF{2}^n\to \GF{2}$, $(v\cdot f)(x)=v\cdot f(x)$ is called a \textit{component Boolean function} of $f$. The \textit{nonlinearity} of $f$ is the minimum distance between its component Boolean functions and affine Boolean functions:
\[\nl(f)=\min_{\substack{v\in \GF{2}^m\setminus\{0\} \\ \alpha\in \mathcal{A}_{n,1}}} d_H(v\cdot f,\alpha).\]
By the covering radius bound, $\nl(f)\leq 2^{n-1}-2^{n/2-1}$. The functions achieving this bound are called \textit{$(n, m)$-bent functions,} or \textit{vectorial bent functions,} or \textit{bent vectorial functions.} It is known that $(n, m)$-bent functions exist only if $n$ is even and $m \leq n/2$. 

The \textit{Walsh transform} of an $(n,m)$-function $f$ is defined as
\[W_f(a,b)=\sum_{x \in \GF{2}^n} (-1)^{b\cdot f(x)+a \cdot x},\]
where $a\in \GF{2}^n$, $b\in \GF{2}^m$. For the nonlinearity, we have
\begin{align*}
\nl(f)&=2^{n-1} - \frac{1}{2} \max_{\substack{a\in \GF{2}^n \\ b\in \GF{2}^m\setminus\{0\}}} |W_f(a,b)|\\
&= 2^{n-1} - \frac{1}{2} \max_{\substack{(a,b)\in \GF{2}^m\times \GF{2}^n\\ (a,b)\neq (0,0) }} |W_f(a,b)|.
\end{align*}

For any vectorial Boolean function $f:\GF{2}^n\to \GF{2}^m$, $a\in \GF{2}^n\setminus\{0\}$, and $b\in \GF{2}^m$, we set
\[\delta(a,b)=|\{x\in \GF{2}^n \mid f(x+a)+f(x)=b\}|.\]
The \textit{differential uniformity} $\delta_f$ of $f$ is defined as 
\[\delta_f=\max_{\substack{a\in \GF{2}^n\setminus\{0\} \\ b\in \GF{2}^m }} \delta(a,b).\]
We also say that $f$ is \textit{differentially $\delta$-uniform,} if $\delta_f\leq \delta$. $\delta_f$ is a positive even integer. When $\delta_f = 2$, the possible smallest value, $f$ is said to be \textit{almost perfect nonlinear} (APN for short). APN functions are a special class of vectorial Boolean functions with the highest possible resistance to differential cryptanalysis. They play an important role in the design of S-boxes for symmetric-key cryptosystems, and they have many desirable cryptographic properties. However, their implementation can be more complex than other types of functions, due to their larger S-boxes.

There are $6$ known infinite classes of pairs $(n,d)$ such that the power function $f(x)=x^d$ is APN over the finite field $\GF{2^n}$. An \textit{APN Gold function} is a power function of the form $f(x)=x^{2^k+1}$ over $\GF{2^n}$, where $\gcd(n,k)=1$. APN Gold functions take a special place among APN functions.

The main notions of equivalence of Boolean and vectorial Boolean functions are as follows: 
\begin{enumerate}
\item Two $(n,m)$-functions $f,g$ are \textit{linearly equivalent,} if there are invertible linear function $L:\GF{2}^n\to\GF{2}^n$, $L':\GF{2}^m\to\GF{2}^m$ such that $g=L'\circ f \circ L$. 
\item Two $(n,m)$-functions $f,g$ are \textit{affine equivalent,} if there are invertible affine functions $L:\GF{2}^n\to\GF{2}^n$, $L':\GF{2}^m\to\GF{2}^m$ such that $g=L'\circ f \circ L$. 
\item Two $(n,m)$-functions $f,g$ are \textit{extended affine (EA) equivalent,} if $g=L'\circ f \circ L+L''$ holds with invertible affine functions $L:\GF{2}^n\to\GF{2}^n$, $L':\GF{2}^m\to\GF{2}^m$, and arbitrary affine $(n,m)$-function $L''$. 
\item Two $(n,m)$-functions $f,g$ are \textit{CCZ-equivalent,} if there is an invertible affine map $LL:\GF{2}^n\times \GF{2}^m\to\GF{2}^n\times \GF{2}^m$ which maps the graph
\[\mathcal{G}_f=\{(x,f(x)) \mid x\in \GF{2}^n\}\]
of $f$ to the graph
\[\mathcal{G}_g=\{(x,g(x)) \mid x\in \GF{2}^n\}\]
of $g$. 
\end{enumerate}
With fixed $n,m$, each equivalence relation above is a strict particular case of the next one, see page 29 in \cite{carlet2020}. Therefore, any EA-invariant property or parameter is linearly and affine invariant. The differential uniformity, the nonlinearity of Boolean functions, and the nonlinearity of vectorial Boolean functions are CCZ-invariant parameters. There are some classes of $(n,m)$-functions, where CCZ-equivalence reduces to EA-equivalence, see Yoshiara \cite{Yoshiara2011} for quadratic APN functions, and Theorem 1 in \cite{Budaghyan2011} by Budaghyan and Carlet for $(n,m)$-bent functions.

\begin{problem}
Is the parameter $\dA{f}$ of $(n,m)$-functions CCZ-invariant?
\end{problem}

The following useful formula has been communicated to us by Claude Carlet. 

\begin{lemma} \label{lm:dH-formula}
Let $f$ be an $(n,m)$-function, and let $\mathcal{L}_{n,m}$ denote the space of linear $(n,m)$-functions. Then
\begin{align*}
\dA{f} &= 2^n-2^{-m} \max_{\substack{b\in \GF{2}^m \\ L^*\in\mathcal{L}_{m,n}}} \sum_{v \in \GF{2}^m} (-1)^{v\cdot b} \, W_f(L^*(v),v).
\end{align*}
\end{lemma}
\begin{proof}
Let $A(x)=L(x)+b$ be an affine $(n,m)$-function with $L\in \mathcal{L}_{n,m}$. Define the linear $(m,n)$-function $L^*$ by $L^*(v)\cdot x=v\cdot L(x)$ for all $x\in \GF{2}^n$. Then
\begin{align*}
d_H(f,A) &= 2^n-2^{-m}\sum_{\substack{x\in \GF{2}^n \\ v\in \GF{2}^m}} (-1)^{v\cdot (f(x)+A(x))} \\
&= 2^n-2^{-m}\sum_{v \in \GF{2}^m} (-1)^{v\cdot b} \sum_{x \in \GF{2}^n} (-1)^{v\cdot f(x)+v\cdot L(x)} \\
&= 2^n-2^{-m}\sum_{v \in \GF{2}^m} (-1)^{v\cdot b} \sum_{x \in \GF{2}^n} (-1)^{v\cdot f(x)+L^*(v)\cdot x} \\
&=2^n-2^{-m} \sum_{v \in \GF{2}^m} (-1)^{v\cdot b} \, W_f(L^*(v),v).
\end{align*}
The lemma follows from the definition of $\dA{f}$. 
\end{proof}

\section{The distance of $(2m,t)$-bent functions to affine functions}

We recall that $(n,m)$-functions achieving the covering radius bound
\[\nl(f)\leq 2^{n-1}-2^{n/2-1}\]
are called \textit{$(n, m)$-bent functions.} It is known that $(n, m)$-bent functions exist only if $n$ is even and $m \leq n/2$. The function $f$ is $(n,m)$-bent if and only if all its non-trivial component functions $v\cdot f$ are bent ($v\in \GF{2}^m\setminus \{0\}$). Equivalently, $W_f(a,b)=\pm 2^{n/2}$ for $b\neq 0$. 

We present three infinite classes of $(2m,t)$-bent functions. To the best knowledge of the author, the first two are commonly known, and the third one is new. 

\subsection{Maiorana-McFarland type $(n,m)$-bent functions}
The first class is a generalization of the Maiorana-McFarland construction. We take maps $\tau:\GF{2^m}\to \GF{2^t}$, $\pi:\GF{2^m}\to \GF{2^m}$, and $h:\GF{2^m}\to \GF{2^t}$ such that $\tau$ is surjective linear, $\pi$ is invertible, and $h$ is arbitrary. We write
\begin{align} \label{eq:gen-maior-mcf}
f(x,y)=\tau(x \pi(y))+ h(y)
\end{align}
with $x,y\in \GF{2^m}$. Then, the component functions are all Maiorana-McFarland bent functions of the form 
\[f'(x,y)=\mathrm{Tr}_{\GF{2^m}/\GF{2}}(x\pi'(y)) + h'(y).\]
Indeed, for fixed $y$, the map
\[\alpha_y:x \mapsto \mathrm{Tr}_{\GF{2^t}/\GF{2}}(\tau(x\pi(y)))\]
is a linear Boolean function. Hence, there is a unique element $\pi'(y)\in\GF{2^m}$ such that $\alpha_y(x) = \mathrm{Tr}_{\GF{2^m}/\GF{2}}(x\pi'(y))$. It is easy to see that $\pi'$ is a permutation of $\GF{2^m}$. This implies that $f(x,y)$ is $(2m,t)$-bent with $t\leq m$.

\subsection{Dillon type $(2m,t)$-bent functions}
To define the two other classes of $(n,m)$-bent functions, we need the notion of left pre-quasifields. 
\begin{definition}
Let $(Q, +)$ be a finite abelian group and $\star : Q \times Q \to Q$ a binary operation with the following properties:
\begin{enumerate}
\item $x \star 0 = 0 \star x = 0$ for all $x \in Q$.
\item $x \star (y + z) = x \star y + x \star z$ for all $x, y, z \in Q$.
\item For all $0 \neq a \in Q$ the mappings $x \mapsto x \star a$ and $x \mapsto a \star x$ are bijective.
\end{enumerate}
Then $Q = (Q, +, \star )$ is called a \textit{(left) pre-quasifield} or a weak (left) quasifield.
\end{definition}
One knows that $(Q, +)$ has to be an elementary abelian $p$-group, $p$ a prime, i.e. $Q$ is a vector space over $\GF{p}$, and $|Q|=q=p^m$ for some positive integer $m$. One usually identifies $(Q,+)$ with the additive group of the finite field $\GF{q}$. 

The $\PS$ class of bent functions was introduced by Dillon in his PhD thesis \cite{Dillon1974} in 1974. A subclass of the $\PS$ class has been presented in algebraic terms by Wu \cite{Wu2013} and by Wan and Xu \cite{Wan2023}. The right division operation of $(Q,+,\star)$ is defined as
\[\star\frac{x}{y} = \begin{cases}
a & \text{if $y\neq 0$ and $x=a\star y$}, \\
0 & \text{if $y=0$.}
\end{cases}
\]

\begin{theorem}[{Theorem 3.1 of \cite{Wu2013} and Corollary 17 of \cite{Wan2023}}]
Let $n=2m$ and $(\GF{2^m},+,\star)$ be a left pre-quasifield. Assume $g$ is a balanced Boolean function in $m$ variables with $g(0)=0$. Then the function
\begin{align} \label{eq:PS-bent}
f(x,y)=g\left(\star\frac{x}{y}\right)
\end{align}
is a $\PS^-$ bent function. 
\end{theorem}

The generalization to $(2m,t)$-bent functions ($t\leq m$) is straightforward. 

\begin{proposition}
Let $t\leq m$ be positive integers and $(\GF{2^m},+,\star)$ be a left pre-quasifield. Assume $\gamma$ is a balanced $(m,t)$-function. Then
\begin{align} \label{eq:PS-nm-bent}
f(x,y)=\gamma\left(\star\frac{x}{y}\right)
\end{align}
is a $(2m,t)$-bent function. 
\end{proposition}
\begin{proof}
Since the composition $\mathrm{Tr}_{\GF{2^t}/\GF{2}} \circ \gamma$ is balanced, all component functions of $f$ are bent functions of type \eqref{eq:PS-bent}.
\end{proof}

\subsection{Quasifield type $(2m,t)$-bent functions}

The following lemma is a well-known property of pre-quasifields.
\begin{lemma} \label{lm:quasiprop}
Let $m$ be a positive integer, $p$ a prime, and $(Q,+,\star)$ a left pre-quasifield with $|Q|=p^m$. If $y_1,y_2\in Q$ are different elements, then the map
\[\beta:Q\to Q, \qquad x\mapsto y_1\star x - y_2\star x\]
is a $\GF{p}$-linear isomorphism. 
\end{lemma}
\begin{proof}
If $x\neq 0$ is in $\ker(\beta)$, then $y_1\star x = y_2\star x$ contradicts the defining property (3) of a left pre-quasifield. Hence, $\beta$ is bijective. The $\GF{p}$-linearity is clear.
\end{proof}

\begin{theorem}
Let $t\leq m$ be positive integers and $(\GF{2^m},+,\star)$ be a left pre-quasifield. Let the maps $\tau:\GF{2^m}\to \GF{2^t}$, $\sigma:\GF{2^m}\to \GF{2^m}$ and $h:\GF{2^m}\to \GF{2^t}$ be such that $\tau$ is surjective linear, $\sigma$ is invertible, and $h$ is arbitrary. Then
\begin{align} \label{eq:quasifield-bent}
f(x,y)=\tau\left(\sigma(y) \star{x}\right) +h(y)
\end{align}
is a $(2m,t)$-bent function. 
\end{theorem}
\begin{proof}
Any component function of $f$ has the shape $\mathrm{Tr}_{\GF{2^t}/\GF{2}}(uf(x,y))$ for some $u\in \GF{2^t}^*$. For fixed $u$ and $y$, the map
\[\alpha_y:x \mapsto \mathrm{Tr}_{\GF{2^t}/\GF{2}}(u\tau(\sigma(y) \star x) +uh(y))\]
is an affine Boolean function in $m$ variables. Hence, there is a unique element $\sigma'(y)\in\GF{2^m}$ such that  
\[\alpha_y(x) = \mathrm{Tr}_{\GF{2^m}/\GF{2}}(x\sigma'(y))+h'(y), \]
where $h'(y)=\mathrm{Tr}_{\GF{2^t}/\GF{2}}(uh(y)) \in \GF{2}$. 

We show that $\sigma'$ is bijective on $\GF{2^m}$. Let us fix $u,y_1,y_2\in \GF{2^m}$, $y_1\neq y_2$. There is an element $v\in \GF{2^t}$ such that \[\mathrm{Tr}_{\GF{2^t}/\GF{2}}(uv)\neq h'(y_1)+h'(y_2).\]
As $\tau$ is surjective, there is a $z\in \GF{2^m}$ with $\tau(z)=v$. By Lemma \ref{lm:quasiprop}, there is $x\in \GF{2^m}$ with $y_1\star x + y_2\star x=z$. Then
\begin{align*}
\alpha_{y_1}(x)+\alpha_{y_2}(x) &= \mathrm{Tr}_{\GF{2^t}/\GF{2}}(u\tau(y_1\star x + y_2\star x) + uh(y_1)+uh(y_2)) \\
&= \mathrm{Tr}_{\GF{2^t}/\GF{2}}(u\tau(z)) + h'(y_1)+h'(y_2) \\
&= \mathrm{Tr}_{\GF{2^t}/\GF{2}}(uv) + h'(y_1)+h'(y_2)\\
&\neq 0.
\end{align*}
This implies $\alpha_{y_1}(x)\neq \alpha_{y_2}(x)$ and $\sigma'(y_1)\neq \sigma'(y_2)$. It follows that all component functions of $f$ are of Maiorana-McFarland type bent function. This finishes the proof of the theorem.  
\end{proof}

In this paper, we do not study the question of EA-equivalence of the $(2m,t)$-bent functions defined above. In general, this is a very difficult question. We only remark that Weng, Feng and Qui \cite{Weng2007} proved that most of the $\PS$ type bent functions, obtained from a Desarguesian spread are not EA-equivalent to any Maiorana–McFarland bent function. This leads us to conclude that, typically, for $t>1$, the $(2m,t)$-bent functions described by \eqref{eq:quasifield-bent} are generally not EA-equivalent to the other two classes, as specified in \eqref{eq:gen-maior-mcf} and \eqref{eq:PS-nm-bent}.

\subsection{Proof of Theorem \ref{thm:main-bent}}

Let us recall the Carlet-Ding-Yuan bound \eqref{eq:CDY-2005} for the distance between affine and $(n,m)$-bent functions:
\begin{align*} 
\left(1-\frac{1}{2^m}\right)\big(2^n-2^{n/2}\big)\leq \dA{f} \leq \left(1-\frac{1}{2^m}\right)\big(2^n+2^{n/2}\big).
\end{align*}
For an $(n,m)$-bent function $f$, the Walsh coefficients are 
\[W_f(a,b)=\begin{cases}
\pm 2^{n/2} & \text{if $b\neq 0$,} \\
0 & \text{if $a\neq 0$, $b=0$,} \\
2^n & \text{if $a=0$, $b=0$.}
\end{cases}\]
Hence, the Carlet-Ding-Yuan bound follows from Lemma \ref{lm:dH-formula} easily. The Liu-Mesnager-Chen Conjecture implies that the true value of $\dA{f}$ is $\left(1-\frac{1}{2^m}\right)\big(2^n-2^{n/2}\big)$. Theorem \ref{thm:main-bent} claims that this holds for two classes of $(n,m)$-bent functions. 

\begin{proof}[Proof of Theorem \ref{thm:main-bent}]
Let $E_i$ be the set of pairs $(x,y)\in \GF{2^m}^2$ such that $f_i(x,y)=f_i(0,0)$, $i=1,2$. We show that $|E_i|=2^{2m-t}+2^m-2^{m-t}$, which implies that $f_i$ has Hamming distance $(1-2^{-t})(2^{2m}-2^m)$ from the constant function $f_i(0,0)$. Therefore, $\dA{f_i}\leq (1-2^{-t})(2^{2m}-2^m)$, and the theorem follows from the Carlet-Ding-Yuan bound. 

Let $T$ be the set of elements $z\in\GF{q^m}$ with $\gamma(z)=f_1(0,0)$. Since $\gamma$ is balanced, $|T|=2^{m-t}$, and $f_1(x,y)=f_1(0,0)$ if and only if $\star\frac{x}{y} \in T$. Moreover, since $\star\frac{0}{0}=0$, we have $0\in T$. The number of solutions of $\star\frac{x}{y}=0$ is $2\cdot 2^m-1$, and the number of solutions of $\star\frac{x}{y}=t\in T\setminus\{0\}$ is $2^m-1$. This implies
\[|E_1|=|f_1^{-1}(f_1(0,0))| = 2\cdot 2^m-1+(2^{m-t}-1)(2^m-1).\]
The same argument applies to $f_2$. In this case, $f_2(0,0)=h(0)$, and the number of solutions of $\sigma(y)\star x=t$ is $2\cdot 2^m-1$ or $2^m-1$, depending on $t=0$ or $t\neq 0$. 
\end{proof}

\section{The distance of monomial functions to affine functions}

In this section, we prove Theorem \ref{thm:main-dH-bounds} in several steps. All functions we are dealing with are $(n,n)$-functions, and $\mathcal{G}_f$ denotes the graph of the function $f$. 

\begin{definition}
Let $A$ be a finite abelian group. We say that $S\subseteq A$ is a \textit{Sidon set} in $A$, if for any $x,y,z,w\in S$ of which at least three are different,
\[x+y\neq z+w.\]
Equivalently, $x-z\neq w-y$. 
\end{definition}

The behavior of Sidon sets depends very much on the structure of $A$, in particular, on the number of involutions in $A$. In this paper, we consider Sidon sets in the elementary Abelian $2$-group $\GF{2}^n$. Moreover, the Sidon property is affine invariant, hence, we can speak of Sidon sets in affine subspaces. (Affine subspaces are translates of $\GF{2}$-linear subspaces.)

\begin{lemma}[Obvious upper bound] \label{lm:obvious}
For a Sidon $S\subseteq \GF{2}^n$, we have 
\[|S|\leq \sqrt{2}\cdot 2^{n/2} +\frac{1}{2}. \]
\end{lemma}
\begin{proof}
See Lemma 4 and 5 of \cite{Nagy2022}.
\end{proof}

The following lemma is folklore:

\begin{lemma} \label{lm:ANPgraph}
The $(n,n)$-function is APN if and only if its graph $\mathcal{G}_f$ is a Sidon set in $\GF{2^n}^{2}$. \qed
\end{lemma}

\begin{lemma} \label{lm:APN-lower}
If $f$ is APN and $A$ is affine, then $|\mathcal{G}_f\cap \mathcal{G}_A|\leq \sqrt{2}\cdot 2^{n/2} +\frac{1}{2}$. 
\end{lemma}
\begin{proof}
The graph of an affine function is an affine subspace of (affine) dimension $n$. As $f$ is APN, and the Sidon property is affine invariant, $\mathcal{G}_f$ is Sidon in $\mathcal{G}_A$. The claim follows from Lemma \ref{lm:obvious}. 
\end{proof}

\begin{lemma} \label{lm:inverse-lower}
Let $n$ be even, $f(x)=x^{2^n-2}$ the Field Inverse Function, and $A$ an affine function. Then $|\mathcal{G}_f\cap \mathcal{G}_A|\leq \sqrt{2}\cdot 2^{n/2} +\frac{3}{2}$. 
\end{lemma}
\begin{proof}
The graph of $f$ is $H\cup \{(0,0)\}$, where $H=\{(x,1/x) \mid x\in \GF{2^n}^*\}$ is the hyperbola $XY=1$. The hyperbola is a Sidon set by Proposition 12 in \cite{Nagy2022}. If $A(0)\neq 0$, then $\mathcal{G}_f\cap \mathcal{G}_A = H \cap \mathcal{G}_A$ is a Sidon set in $\mathcal{G}_A$. If $A(0)=0$, then $\mathcal{G}_f\cap \mathcal{G}_A = (H \cap \mathcal{G}_A) \cup \{(0,0)\}$, one larger than the Sidon set $H \cap \mathcal{G}_A$. In both cases, the claim follows from Lemma \ref{lm:obvious}. 
\end{proof}

\begin{proof}[Proof of Theorem \ref{thm:main-dH-bounds}]
Since the Hamming distance of two functions $f,g$ is
\[d_H(f,g)=2^n-|\mathcal{G}_f\cap \mathcal{G}_g|,\]
Lemmas \ref{lm:APN-lower} and \ref{lm:inverse-lower} imply the lower bounds for $\dA{h_1}$ and $\dA{h_2}$. To prove the upper bounds, it suffices to present affine functions $A_1,A_2$ such that
\[|\mathcal{G}_{h_i}\cap \mathcal{G}_{A_i}|=2^{n/2}+2, \qquad (i=1,2).\]
With $A_1(x)=x^{2^{n/2}}$ we have $h_1(x)=A_1(x)$ if and only if either $x=0$, or $x\neq 0$ and
\begin{align*}
 1/x=x^{2^{n/2}} \Leftrightarrow  1=x^{2^{n/2}+1}.
\end{align*}
As $2^{n/2}+1$ divides $2^n-1$, the equation $1=x^{2^{n/2}+1}$ has $2^{n/2}+1$ solutions in $\GF{2^n}$. All together, $h_1(x)=A_1(x)$ has $2^{n/2}+2$ solutions in $\GF{2^n}$. 

Similar argument applies for $h_2(x)$ with $A_2(x)=x^{2^{n-1}}=x^{\frac{1}{2}}$. Then 
\begin{align*}
x^{\frac{1}{2}} = x^{2^{n/2-1}+1} & \Leftrightarrow x=x^{2^{n/2}+2} \\
& \Leftrightarrow x=0 \text{ or } 1=x^{2^{n/2}+1}. 
\end{align*}
Again, $1=x^{2^{n/2}+1}$ has $2^{n/2}+1$ solutions, and $h_2(x)=A_2(x)$ has $2^{n/2}+2$ solutions in $\GF{2^n}$. 
\end{proof}

\section{Distance results in small dimensions}

In this section, we present computational results on the distance between APN functions and affine functions in dimension $n\leq 9$. The results for $n\leq 5$ have been obtained by Ryabov \cite{Ryabov2023}. Our approach is based on the detailed study of the Sidon sets $\mathcal{G}_f$ in $\GF{2}^{2n}$, and $\mathcal{G}_f\cap \mathcal{G}_A$ in $\GF{2}^n$, where $f$ is an APN function in dimension $n$, and $A$ is affine. We use a computational method to compute automorphisms and isomorphisms of Sidon sets. For Sidon sets, this method is novel, but it has many similarities with Kaleyski's families of invariants \cite{Kaleyski2021} for deciding EA-equivalence of $(n,m)$-functions.

\subsection{APN functions in small dimension}
We give a brief overview of the known APN functions in dimension $n\leq 9$. APN power functions play a distinguished role in the theory of Boolean functions. The six known infinite families of APN power functions are given in Table \ref{tab:apn-power}.

\begin{table} 
\caption{(Table 11.2 in \cite{carlet2020}) Known infinite families of APN power functions over $\GF{2^n}$.\label{tab:apn-power}}
\begin{tabular}{lll}
\hline\hline
Functions \hspace{2cm} & Exponents $d$ \hspace{4cm} & Conditions \\
\hline
Gold & $2^{i+1}$ & $\gcd(i,n)=1$ \\
Kasami & $2^{2i}-2^i+1$ & $\gcd(i,n)=1$ \\
Welch & $2^t+3$ & $n=2t+1$ \\
Niho & $2^t+t^{t/2}-1$, $t$ even & $n=2t+1$ \\
 & $2^t+t^{(3t+1)/2}-1$, $t$ odd &  \\
Inverse & $2^{2t}-1$ & $n=2t+1$ \\
Dobbertin & $2^{4t}+2^{3t}+2^{2t}+2^t-1$ & $n=5t$ \\
\hline\hline
\end{tabular}
\end{table}

\begin{theorem}[{Theorem 5 in \cite{Brinkmann2008}}]
There are $2$ APN functions in dimension $4$ up to EA-equivalence. One of those is EA-equivalent to the power function $x^3$. There are $7$ APN functions in dimension $5$ up to EA-equivalence. Five of those are EA-equivalent to the power functions $x^3$, $x^5$, $x^7$, $x^{11}$, $x^{15}$.
\end{theorem}

\begin{theorem}[{Table 1 in \cite{Calderini2020}}]
In dimension $6$ there are $14$ known APN functions ($13$ are quadratics) up to CCZ-equivalence. 
\end{theorem}

In 2009, Dillon et al. \cite{Browning2010} found an APN permutation in dimension $6$, which is the first APN permutation in even dimension. Dillon's APN permutation is CCZ-equivalent to the Kim function $x^3 + \zeta x^{24} + x^{10}$, $\langle \zeta \rangle = \GF{2^6}^*$. 

In dimension $7$, there are $488$ quadratic APN functions. The classification of quadratic APN functions is complete up to EA-equivalence, see \cite{Kalgin2022} by Kalgin and Idrisova. Known non-quadratic APN functions are the Kasami functions $x^{13}$, $x^{57}$, and the Field Inverse monomial function $x^{126}$. 

In dimension $8$, there are thousands of known APN functions, but a complete classification seems out of reach. Therefore for $n=8$, we restrict our interest to APN power functions. The situation is even more involved for $n=9$, hence, we focus on Gold APN power functions.

\begin{theorem}[Table 1.7 in \cite{Maxwell2005}]
\begin{enumerate}[(i)]
\item There are $18$ APN power functions in dimension $8$, all are EA-equivalent to the power function $x^3$, $x^9$ or $x^{57}$. 
\item There are $6$ Gold APN power functions in dimension $9$, all are EA-equivalent to the power function $x^3$, $x^5$ or $x^{17}$. 
\end{enumerate}
\end{theorem}

We remark that while the results of \cite{Maxwell2005} are computational for small $n$, Yoshiara \cite{Yoshiara2016} studies the equivalence problem of APN power functions with APN power or quadratic functions from a theoretical point of view. 

\subsection{Sidon sets in small dimension}

In a vector space $V$ of dimension $n$, an \textit{affine basis} is a set $v_0,\ldots,v_n$ such that all $x\in V$ can be written as $x=\sum_{i=0}^n c_i v_i$ in a unique way, with $\sum_{i=0}^nc_i=1$. An \textit{extended affine basis} is obtained from an affine basis by adding $\sum_{i=0}^nv_i$ to it. An affine basis is always Sidon and an extended affine basis is Sidon if $n\geq 4$. 

Let $S$ be a Sidon set in the group $A$. We say that $S$ is \textit{complete} (or \textit{maximal}) if it is not contained in a strictly larger Sidon set. In Table \ref{tab:small-sidon}, we list all complete Sidon sets of $\GF{2}^n$, $2\leq n\leq 8$. For $n=9$, we know that the largest Sidon set is unique of size $24$. Moreover, using a computer one can easily construct many complete Sidon sets of order $21,22,23$. The existence of these Sidon sets is not new, they were listed in Proposition 2.8 of \cite{Czerwinski2023} by Czerwinski and Pott for $n\leq 8$. In the language of binary linear codes (see Lemma 9 in \cite{Nagy2022}), Wagner \cite{Wagner1965} constructed the Sidon set of size $24$ in $\GF{2}^9$; the uniqueness has been shown by Simonis \cite{Simonis2000}. 

\begin{table}
\caption{Complete Sidon sets in dimension $2$ to $9$. \label{tab:small-sidon}}
\begin{tabular}{|c|c|p{7cm}|c|}
\rule[-1ex]{0pt}{3.5ex} Dimension & Size & \centering Structure & Aut. group \\
\hline\hline
\rule[-1ex]{0pt}{3.5ex} 2 & 3 & affine basis & $S_3$ \\
\hline
\rule[-1ex]{0pt}{3.5ex} 3 & 4 & affine basis & $S_4$ \\
\hline
\rule[-1ex]{0pt}{3.5ex} 4 & 6 & extended affine basis & $S_6$ \\
\hline
\rule[-1ex]{0pt}{3.5ex} 5 & 7 & extended affine basis & $S_6$ \\
\hline
\rule[-1ex]{0pt}{3.5ex} 6 & 8 & extended affine basis = graph of $x^3$ over $\GF{8}$ & $S_8$ \\
\rule[-1ex]{0pt}{3.5ex}  & 9 & ellipse in $\GF{8}^2$ & $S_3 \wr S_3$ \\
\hline
\rule[-1ex]{0pt}{3.5ex} 7 & 12 & transitive & $2^4:\Omega^+(4,2)$ \\
\hline
\rule[-1ex]{0pt}{3.5ex} 8 & 15 & hyperbola in $\GF{16}^2$ & $S_3\times S_3$ \\
\rule[-1ex]{0pt}{3.5ex}  & 16 & graph of $x^3$ over $\GF{16}$ & $\mathrm{A\Gamma L}(2,4)$ \\
\rule[-1ex]{0pt}{3.5ex}  & 16 & orbit lengths $2, 2, 12$ & $C_2\times S_4$ \\
\rule[-1ex]{0pt}{3.5ex}  & 18 & ellipse in $\GF{16}^2$ & $\mathrm{PSL}(2,17)$ \\
\hline
\rule[-1ex]{0pt}{3.5ex} 9 & 24 & $\mathrm{SL}(2,3)$ acts regularly & $\mathrm{SL}(2,3) \rtimes C_4$ \\
\rule[-1ex]{0pt}{3.5ex}  & $21,22,23$ & many complete Sidon sets & -- \\
\hline 
\end{tabular}
\end{table}

In the Appendix, we explain some details about the computer calculations giving the results above. 

\subsection{Graphs of APN functions and gerbera configurations}

\begin{theorem}[{Section 3 in \cite{Ryabov2023}}]
If $n\in \{1,\ldots,5\}$, then all APN functions in $n$ variables have the same distance to affine functions:
\begin{center}
\begin{tabular}{c|c|c|c|c|c}
$n$ & $1$ & $2$ & $3$ & $4$ & $5$ \\
\hline
$\dA{f}$ & $0$ & $1$ & $4$ & $10$ & $25$ \\
\end{tabular}
\end{center}
If $n=6$ and $f$ is Dillon's APN permutation, then $\dA{f}=55$. 
\end{theorem}

We see that for $n\leq 5$, the Liu-Mesnager-Chen Conjecture holds. Comparing the values $2^n-\dA{f}$ with the maximal Sidon set sizes of Table \ref{tab:small-sidon}, we obtain that for $n\leq 5$, and for Dillon's APN permutation, there is an affine function $\alpha$ such that $\mathcal{G}_f\cap \mathcal{G}_\alpha$ is a maximum size Sidon set in $\mathcal{G}_\alpha\cong \GF{2}^n$. 

\begin{theorem} \label{thm:n6789}
We have the following values for the parameter $\dA{f}$.
\begin{enumerate}[(i)]
\item $n=6$:
\[\dA{f}=55=2^6-9\]
for all representatives of the known $14$ CCZ-equivalence classes. 
\item $n=7$: 
\[\dA{x^3}=\dA{x^9}=117=2^7-11\]
for two Gold APN power functions. For the remaining $489$ APN functions $f$, we have
\[\dA{f}=116=2^7-12. \]
\item $n=8$: 
\[\dA{x^9}=238=2^8-18, \]
and for all other APN exponents $d$, 
\[\dA{x^d}\geq 240=2^8-16. \]
\item $n=9$:
\[\dA{x^d}\geq 491 =2^9-21\]
for all Gold exponents $d=3, 5, 9, 17, 33, 103, 171$. 
\end{enumerate}
In particular, the Liu-Mesnager-Chen Conjecture does not hold for $n=8,9$. 
\end{theorem}

The time complexity of brute force computation of the distance $\dA{f}$ is $O(2^{n^2+2n})$ for an $(n,n)$-function $f$. To obtain Theorem \ref{thm:n6789}, we speed up the computation using the notion of gerbera configurations. 

\begin{definition}
In $\GF{2}^n$, a \textit{$w$-centered  gerbera configuration} is a collection $\mathcal{T}$ of subsets $T_i=\{x_i,y_i,z_i\} \subseteq \GF{2}^n$, $i=1,\ldots,t$, such that $|T_i|=3$ and $x_i+y_i+z_i=w$ for all $i$. The integer $t$ is the \textit{size} of the gerbera configuration. The sets $T_i$ are the \textit{leaves} of $\mathcal{T}$. 

We say that the gerbera configuration is contained in the set $S$ if $T_i\subseteq S$ for all $i$. The affine span of a gerbera configuration is the smallest affine subspace containing it.
\end{definition}

\begin{lemma} \label{lm:small-gerbera}
Let $S$ be a Sidon set in $\GF{2}^n$
\begin{enumerate}[(i)]
\item If $(n,|S|)\in \{(6,9),(7,12)\}$, then $S$ contains gerbera configurations of size $3$, and $S$ does not contain gerbera configurations of size $4$.
\item If $(n,|S|)\in \{(8,17),(8,18),(9,22),(9,23),(9,24)\}$, then $S$ contains gerbera configurations of size $4$. 
\end{enumerate}
\end{lemma}
\begin{proof}
For $n=6,7,8$, one proves by a straightforward calculation using the classification of complete Sidon sets; see Table \ref{tab:small-sidon}. For $n=9$, a simple counting argument applies. 
\end{proof}

\begin{lemma} \label{lm:gerbera-dim}
Let $S$ be a Sidon set in $\GF{2}^n$, and let $\mathcal{T}$ be a gerbera configuration of size $t$, contained in $S$. Then the following hold:
\begin{enumerate}[(i)]
\item The leaves of $\mathcal{T}$ are disjoint.
\item The affine dimension of $\mathcal{T}$ is at most $2t$.
\item If $t\leq 4$, then $\mathcal{T}$ has affine dimension $2t$. 
\end{enumerate}
\end{lemma}
\begin{proof}
(i) Assume $x_i=x_j$ with $i\neq j$. Then $y_i+z_i=y_j+z_j$, which implies $T_i=T_j$ by $y_i\neq z_i$. (ii) The affine span of $\mathcal{T}$ is the affine span of the $2t+1$ points $w,x_1,y_1,\ldots,x_t,y_t$. (iii) If $t\leq 3$, then $\GF{2}^{2t-1}$ does not contain a Sidon set of size $3t$. By Lemma \ref{lm:small-gerbera}(i), $t=4$ is not possible in affine dimension $7$ or less. 
\end{proof}

In the remainder of this section, $f:\GF{2^n}\to \GF{2^n}$ is an APN function. We consider its graph $\mathcal{G}_f$ as a subset of $W=\GF{2}^{2n}\oplus 1$; $\mathcal{G}_f$ is a Sidon set in $W$ (Lemma \ref{lm:obvious}). For any affine subspace $\Sigma$ of $W$, $\mathcal{G}_f\cap \Sigma$ is Sidon in $\Sigma$. 

We define the affine subspace $\Pi_0=\{(x_1,\ldots,x_n,0,\ldots,0,1) \mid x_i\in \GF{2}\}$ of $W$, and the projection $\mathrm{proj}:W\to \Pi_0$,
\[(x_1,\ldots,x_n,x_{n+1},\ldots,x_{2n},1) \mapsto (x_1,\ldots,x_n,0,\ldots,0,1).\]
We denote by $\mathcal{S}$ the set of $n$-dimensional affine subspaces $\Sigma$ of $W$ with $\mathrm{proj}(\Sigma)=\Pi_0$. $\mathcal{S}$ is precisely the set of subspaces of the form $\mathcal{G}_A=\{(x,A(x),1) \mid x\in \GF{2}^n\}$, where $A$ is an affine $(n,n)$-function. 

\begin{lemma}
$\dA{f}\leq 2^n-s$ if and only if there is a $\Sigma \in \mathcal{S}$ such that $|\Sigma \cap \mathcal{G}_f|\geq s$. 
\end{lemma}
\begin{proof}
$2^n-\dA{f}=\max_{A\in \mathcal{A}} |\mathcal{G}_f\cap \mathcal{G}_A| = \max_{\Sigma \in \mathcal{S}} |\Sigma \cap \mathcal{G}_f|$. 
\end{proof}

\begin{proof}[Proof of Theorem \ref{thm:n6789}]
Our strategy is to scan all $w$-centered gerbera configurations in $\mathcal{G}_f$ of size $t$, where $w\in W\setminus \mathcal{G}_f$ fixed, and $t=3$ if $n=6,7$, and $t=4$ for $n=8,9$. If $\mathcal{T}$ is such a gerbera configuration and $\Sigma$ is the affine span of $\mathcal{T}$, then we define $Y$ as the set of all affine subspaces of dimensions $n$ containing $\Sigma$. If $n=6,8$, then $\Sigma$ itself has affine dimension $n$ by Lemma \ref{lm:gerbera-dim}, and $Y=\{\Sigma\}$. For all $\Sigma'\in Y$ we check if $\Sigma'\cap \mathcal{G}_f$ has the appropriate size, and $\mathrm{proj}(\Sigma')=\Pi_0$. The "appropriate size" follows from the claims and the fact that $\Sigma'\cap \mathcal{G}_f$ is a Sidon set in $\Sigma'$. 

For fixed $w$, these searches take a few minutes. If $n=6$, then we can do the computation for all $w\in W\setminus \mathcal{G}_f$. For $n=7$, $f\neq x^3,x^9$, almost any choice of $w$ will give a $\Sigma'$ with $|\Sigma'\cap \mathcal{G}_f|= 12$, proving $\dA{f}=2^7-12$. The Gold APN functions $x^3,x^9$ are monomial and quadratic, and their automorphism groups have two orbits on $W\setminus \mathcal{G}_f$. Hence, one has to search for $\Sigma'$ only for two values of $w$. Similarly, if $n=8,9$ and $f$ is a Gold monomial APN function, then the automorphism group of $\mathcal{G}_f$ has very few orbits on $W\setminus \mathcal{G}_f$. The most challenging case is when $n=8$ and $f(x)=x^{57}$ is Kasami's APN function. Then, the number of orbits is $35$, and the computation takes a few hours. 
\end{proof}

\section{Conclusion}

In this paper, we have made significant progress on the study of vectorial Boolean functions and their nonlinear properties. We proved that for $(2m,t)$-bent functions constructed using pre-quasifields, the distance to affine functions can be precisely determined (Theorem \ref{thm:main-bent}). Furthermore, we obtained tight bounds on the distance to affine functions for specific APN functions, including Nyberg's Field Inverse Function and a monomial APN function of Gold type (Theorem \ref{thm:main-dH-bounds}). These results provide strong evidence for the Liu-Mesnager-Chen Conjecture in the case of vectorial bent functions, and demonstrate its validity for certain classes APN of functions. Our findings have important implications for cryptography, particularly in the design of secure cryptographic primitives. Future research directions may include extending our results to other classes of functions and exploring the connections between vectorial Boolean functions, Sidon sets, and other areas of mathematics.

\appendix

\section{Algorithms for Sidon sets in $\GF{2}^n$}

We recall that the Sidon property is affine invariant. We say that the Sidon sets $S_1,S_2$ of $V$ are \textit{isomorphic,} if there is an invertible affine map of $V$, which maps $S_1$ to $S_2$. \textit{Automorphisms} of the Sidon set $S$ are isomorphisms onto itself. 

Let $W$ be a vector space over $\GF{2}$. To represent Sidon sets in $W$, we embed $W$ in $W\oplus \GF{2}$ by $\epsilon: x\mapsto (x,1)$. $\epsilon$ induces an isomorphism between Sidon sets of $W$ and Sidon sets contained in $W\oplus 1$. An affine map $x\mapsto xA+b$ will correspond to the linear transformation with matrix $\begin{bmatrix} A&0\\b&1 \end{bmatrix}$. Since we are interested in complete Sidon sets $S$, we may (and will) assume that the affine span of $S$ is $W\oplus 1$. In this case, isomorphisms and automorphisms $S$ are uniquely defined by their permutation action on $S$. 

In the computer, Sidon sets are ordered lists. Therefore we consider \textit{ordered} Sidon sets, where $S[i]$ refers to the $i$th element, $i=1,\ldots,|S|$. We denote by $H_S$ the $|S|\times (n+1)$ matrix whose rows are the elements of $S$. The last column of $H_S$ is the all-one column vector $\mathbf{j}$. The column space of $H_S$ is denoted by $\mathcal{U}_S$. The rank of $H_S$ is $n+1$, and $\mathcal{U}_S$ is an $(n+1)$-dimensional subspace of the space $(\GF{2}^{|S|})^\top$ of column vectors. For a permutation $\sigma$ of $\{1,\ldots,|S|\}$, $P_\sigma$ denotes its permutation matrix. 

\begin{lemma} \label{lm:perm-to-iso}
Let $S_1,S_2$ be a full rank Sidon sets of $\GF{2}^n\oplus 1$, $s=|S_1|=|S_2|$, $A$ an $(n+1)\times (n+1)$ matrix, and $\sigma$ a permutation of $\{1,\ldots,s\}$.
\begin{enumerate}[(i)]
\item The $(n+1)\times (n+1)$ matrix $A$ is an isomorphism $S_1\to S_2$ if and only if there is a permutation $\sigma$ of $\{1,\ldots,s\}$ such that $H_{S_1}A=P_\sigma H_{S_2}$. The left action of $P_\sigma^\top$ maps $\mathcal{U}_{S_1}$ to $\mathcal{U}_{S_2}$. 
\item The permutation $\sigma$ is induced by an isomorphism $S_1\to S_2$ if and only if the left action of $P_\sigma^\top$ maps $\mathcal{U}_{S_1}$ to $\mathcal{U}_{S_2}$. 
\end{enumerate}
\end{lemma}
\begin{proof}
(i) and the ``only if'' part of (ii) are obvious. Let $\sigma$ be a permutation such that $P_\sigma^\top$ maps $\mathcal{U}_{S_1}$ to $\mathcal{U}_{S_2}$. Then, the column spaces of $P_\sigma^\top H_{S_1}$ and $H_{S_2}$ are equal, and there is a matrix $B$ with $P_\sigma^\top H_{S_1} B= H_{S_2}$. Let $(b_1,\ldots,b_{n+1})^\top$ be the last column of $B$. Since both $P_\sigma^\top H_{S_1}$ and $H_{S_2}$ have last column $\mathbf{j}$, and the matrices have full rank $n+1$, we have $b_1=\cdots=b_n=0$ and $b_{n+1}=1$. Hence, $B$ has shape $\begin{bmatrix} A&0\\b&1 \end{bmatrix}$, and it is an isomorphism $S_1\to S_2$. 
\end{proof}

When the permutation $\sigma$ is given, the previous lemma enables a fast check if it induces a Sidon set isomorphism. To reduce the number of permutations to be checked, we assign a bipartite graph to the Sidon set in an invariant way. Notice that in our representation, for any subset $B$ of a Sidon set $S$, $1\leq |B|\leq 5$ implies $\sum_{x\in B} x \neq 0$. 

\begin{definition}
Let $S$ be a Sidon set of $\GF{2}^n\oplus 1$. We define 
\[\mathcal{D}_S = \{ B\subseteq S \mid |B|=6, \; \sum_{x\in B} x=0\},\]
and the bipartite graph $\Gamma_S$ with vertex set $V(\Gamma_S)=S\cup \mathcal{D}_S$, and edge set 
\[E(\Gamma_S)=\{\{x,B\} \mid x\in B\}.\]
\end{definition}

\begin{lemma}
Let $S_1,S_2$ be a full rank Sidon sets of $\GF{2}^n\oplus 1$. An isomorphism $\alpha:S_1\to S_2$ induces a graph isomorphism $\Gamma_{S_1}\to \Gamma_{S_2}$. \qed
\end{lemma}

The following algorithms reduce the isomorphism problem of Sidon sets to the graph isomorphism problem (GI). Using efficient solvers for the Graph Isomorphism Problem (GI), we can implement a practical method to compute isomorphisms and automorphism of Sidon sets up to dimension 15. 

\begin{algorithm}
\caption{Automorphism group of a full-rank Sidon set in $\GF{2}^n\oplus 1$}\label{alg:sidon-aut}
\begin{algorithmic}[1]
\Function{AutSidon}{$S$}
\State Compute the bipartite graph $\Gamma_S$
\State $G \gets \Aut(\Gamma_S)$ 
\State $\tilde{G} \gets $ the $|S|$-dimensional permutation matrix action of $G$
\State $\tilde{G}_0 \gets $ the stabilizer of $\mathcal{U}_S$ in $\tilde{G}$
\State \Return the $(n+1)$-dimensional matrix group describing the linear action of $\tilde{G}_0$ on $\mathcal{U}_S$
\EndFunction
\end{algorithmic}
\end{algorithm}

Typically, the computation of $\Gamma_S$ takes the most time for both Algorithms \ref{alg:sidon-aut} and \ref{alg:sidon-isom}. The computation of a subspace stabilizer in a matrix group is challenging in general. In our case, $\tilde{G}\cong \Aut(\Gamma_S)$ is a relatively small group, and $\mathcal{U}_S$ has a short $\tilde{G}$-orbit. 

\begin{algorithm} 
\caption{Isomorphisms between full-rank Sidon sets in $\GF{2}^n\oplus 1$} \label{alg:sidon-isom}
\begin{algorithmic}[1]
\Function{IsomSidon}{$S_1,S_2$}
\State Compute the bipartite graphs $\Gamma_{S_1},\Gamma_{S_2}$
\State $\varphi \gets$ an isomorphism $\Gamma_{S_1} \to \Gamma_{S_2}$
\If{$\varphi$ fails}
	\State \Return fail
\EndIf
\State $U \gets $ the image of $\mathcal{U}_{S_1}$ under $P_\varphi$
\State $G \gets \Aut(\Gamma_{S_2})$ 
\State $\tilde{G} \gets $ the $|S|$-dimensional permutation matrix action of $G$
\If{$U$ is in the $\tilde{G}$-orbit of $\mathcal{U}_{S_2}$}
	\State $\tilde{\alpha} \gets$ element of $\tilde{G}$ which maps $U$ to $\mathcal{U}_{S_2}$
	\State $\alpha\gets$ preimage of $\tilde{\alpha}$ in $G$
	\State \Return $\alpha\circ \varphi$
\Else
	\State \Return fail
\EndIf
\EndFunction
\end{algorithmic}
\end{algorithm}

\printbibliography

\end{document}